\documentclass[11pt, oneside]{article}
\usepackage{amsmath}
\usepackage{amssymb}
\usepackage{latexsym}
\usepackage{theorem}
\usepackage{mathrsfs}
\usepackage[all]{xy}
\usepackage{url}

\DeclareMathOperator{\cof}{-cof}
\DeclareMathOperator{\inj}{-inj}
\DeclareMathOperator{\pic}{Pic}

\newcommand{\smashprod}{\wedge}

\newcommand{\co}{\colon \!}

\newcommand{\zz}{{ \mathbb{Z} }}

\newcommand{\qq}{{ \mathbb{Q} }}

 \DeclareMathOperator{\id}{Id}

\DeclareMathOperator{\h}{H} 
\DeclareMathOperator{\Ho}{Ho}

\newcommand{\abcat}{{ \mathcal{G} }}

\DeclareMathOperator{\homSSS}{Hom}
\renewcommand{\hom}{\homSSS}
\DeclareMathOperator{\LimSSS}{lim}
\renewcommand{\lim}{\LimSSS}

\newtheorem{theorem}{Theorem}[section]
\newtheorem{proposition}[theorem]{Proposition}
\newtheorem{corollary}[theorem]{Corollary}
\newtheorem{lemma}[theorem]{Lemma}

\newtheorem{definition}[theorem]{Definition}

\newtheorem{rmk}[theorem]{Remark}

\newenvironment{proof}[1][Proof]{\vskip -0.25cm \textbf{#1} }
{\hfill \rule{0.5em}{0.5em}}

\title{Monoidality of Franke's Exotic Model}
\author{David Barnes\footnote{Partially supported by NSERC}$\phantom{a}^1$\\ 
\footnotesize{The University of Western Ontario} \\ 
\footnotesize{Department of Mathematics} \\ 
\footnotesize{London, Ontario N6A 5B7} \\ 
\footnotesize{Canada} \\
\footnotesize{d.j.barnes@shef.ac.uk} 
\and 
Constanze Roitzheim\thanks{Supported by EPSRC grant EP/G051348/1,  corresponding author}  \\
\footnotesize{The University of Glasgow} \\
\footnotesize{Department of Mathematics} \\ 
\footnotesize{University Gardens, Glasgow G12 8QW} \\ 
\footnotesize{United Kingdom} \\ 
\footnotesize{c.roitzheim@maths.gla.ac.uk} }
\date{July 26, 2011}

\begin{document}
\maketitle
\begin{abstract}
\noindent We discuss the monoidal structure on Franke's algebraic
model for the $K_{(p)}$-local stable homotopy category at odd primes
and show that its Picard group is isomorphic to the integers. \\
\phantom{f}
\end{abstract}

\section*{Introduction}

\footnotetext[1]{Current Address: The University of Sheffield, Hicks Building, Sheffield  S3 7RH, United Kingdom}

A well-established method to study the structure of the stable homotopy category $\Ho(\mathcal{S})$ is chromatic filtration. This concerns Bousfield localisation with respect to homology theories $E(n)$, $n \in \mathbb{N}$ for a fixed prime $p$ absent from notation. The resulting ``chromatic layers'' $\Ho(L_n \mathcal{S})$ provide a better and better approximation of $\Ho(\mathcal{S})$ at $p$ as $n$ increases.

For $n=1$, $E(1)$ is the Adams summand of $p$-local complex $K$-theory, so $\Ho(L_1 \mathcal{S})$ is the $K_{(p)}$-local stable homotopy category. Since $K$-theory has been studied extensively, we can make use of a wealth of tools to study this first chromatic layer. Startlingly, the behaviour of $\Ho(L_1 \mathcal{S})$ at odd primes differs significantly from $p=2$. For $p=2$ all higher homotopy structure of $\Ho(L_1 \mathcal{S})$ is encoded in its triangulated structure, meaning that $\Ho(L_1 \mathcal{S})$ is \emph{rigid} at $p=2$, see \cite{Roi07}. As a consequence, $\Ho(L_1 \mathcal{S})$ at $p=2$ cannot be described by an algebraic model. For odd primes however, the situation is completely different.

For odd primes Jens Franke constructed an algebraic model category $C^{(T,N)}(\mathcal{A})$ whose homotopy category $D^{(T,N)}(\mathcal{A})$ is equivalent to $\Ho(L_1 \mathcal{S})$, \cite{Fra96}. But the underlying models $L_1 \mathcal{S}$ and $C^{(T,N)}(\mathcal{A})$ are not Quillen equivalent by \cite{Fra96} and \cite{Roi08}. As $C^{(T,N)}(\mathcal{A})$ has such a different homotopical behaviour from the standard model $L_1 \mathcal{S}$, it is called an \emph{exotic model}.

Although some general differences between $p=2$ and $p>2$ have been studied in \cite[Section 6]{Roi07}, it is still mysterious how this exotic model came along. Furthermore, it is still almost entirely unknown for which $n$ and $p$,  $\Ho(L_n \mathcal{S})$ possesses exotic models and if it does, how many. It is not even known if Franke's model is the only algebraic model in its range.

One structural tool that has not yet been made use of is monoidality. We are going to focus on it in this paper. After defining a monoidal product on 
$C^{(T,N)}(\mathcal{A})$, we find that Franke's model structure on this category is not monoidal. This means that it does not induce a monoidal structure on $D^{(T,N)}(\mathcal{A})$. So our goal is to construct a new model structure on $C^{(T,N)}(\mathcal{A})$ that is Quillen equivalent to Franke's model while also being compatible with the monoidal product.

Furthermore, we want the newly defined derived product $\wedge^L_{\mathbb{P}\mathcal{I}}$ on $D^{(T,N)}(\mathcal{A})$ to interact reasonably with the smash product $\wedge^L$ on $\Ho(L_1 \mathcal{S})$. Let
\[
\mathcal{R}: D^{(T,N)}(\mathcal{A}) \longrightarrow \Ho(L_1 \mathcal{S})
\]
denote Franke's triangulated equivalence. In \cite{Gan07}, Nora Ganter constructed a natural isomorphism 
\[
\mathcal{R}(C_* \otimes^L_{E(1)_*} D_*) \cong \mathcal{R}(C_*) \wedge^L \mathcal{R}(D_*).
\]
However, that paper does not accurately define the non-derived tensor
product of quasi-periodic cochain complexes so the definition of 
the derived tensor product is not complete. 
 
Our construction is compatible with the assumptions needed for Ganter's result. Hence, it closes a gap in \cite{Gan07} and so allows us to use the isomorphism
\[
\mathcal{R}(C_* \wedge^L_{\mathbb{P}\mathcal{I}} D_*) \cong \mathcal{R}(C_*) \wedge^L \mathcal{R}(D_*)
\]
to relate our new monoidal product $\wedge^L_{\mathbb{P}\mathcal{I}}$ to $\wedge^L$. It should be noted that it is unknown if $\mathcal{R}$ is associative for $p>5$ and it is definitely not associative for $p=3$.

Even further, combining this with work of Hovey and Sadofsky \cite{HovSad99}, we can read off the Picard group of $D^{(T,N)}(\mathcal{A})$. We hope that our results contribute to understanding the concept of exotic models in the future. 

\subsubsection*{Organisation}
In Section \ref{adjoints} we revise the concept of quasi-periodic chain complexes $C^{(T,N)}(\mathcal{G})$. Here, $\mathcal{G}$ denotes a Grothendieck abelian category, $T$ a self-equivalence of $\mathcal{G}$ and $N\ge 0$ the periodicity index. Quasi-periodic chain complexes form the basis of Franke's construction. In particular, they are chain complexes in $\mathcal{G}$. We recall how to create model structures on 
$C^{(T,N)}(\mathcal{G})$ using the forgetful functor to chain complexes $C(\mathcal{G})$. 

In Section \ref{modules}, we explain how $C^{(T,N)}(\mathcal{G})$ can be described as a category of modules over a ring object in $C(\mathcal{G})$. The ring object will be the ``periodified'' unit $\mathbb{P}\mathcal{I}$. 

Section \ref{hopfalgebroids} recalls some definitions and properties about comodules over Hopf algebroids. They are used in Section \ref{Franke}, which concerns Franke's category. Here, we specify the Grothendieck abelian category $\mathcal{A}$, the self-equivalence $T$ and period $N=1$. The abelian category $\mathcal{A}$ is equivalent to $E(1)_*E(1)$-comodules. We then describe the resulting model structure and some of its properties. 

In Franke's case, $\mathcal{A}$ does not have enough projectives, only enough injectives. But the injective model structure is not monoidal. A step towards a solution is the \emph{relative projective model structure} described in Section \ref{relativeprojective}. This was first introduced by Christensen and Hovey in \cite{ChrHov02}. The induced model structure on $C^{(T,1)}(\mathcal{A})$ is monoidal but is not Quillen equivalent to Franke's model as it does not have enough weak equivalences.

We use the above to construct a \emph{quasi-projective model structure} in Section \ref{quasiprojective}. For the construction, we formally add weak equivalences to the relative projective model structure. Eventually we arrive at a model category that is Quillen equivalent to Franke's model and is a monoidal model category. 

Finally, in Section \ref{Picard}, we relate our result to Ganter's theorem and compute the Picard group of the exotic model, $\pic(D^{(T,1)}(\mathcal{A}))$. We further place it in context with other results about the $E(n)$-local stable homotopy category, hopefully shedding some light on the existence of exotic models versus rigidity. 

\bigskip
We would like to thank Andy Baker, Dan Christensen, Nora Ganter and Uli Kr\"ahmer for motivating
discussions.

\section{Quasi-periodic chain complexes}\label{adjoints}

We use $\abcat$ for a general abelian category and reserve
$\mathcal{A}$ for Franke's category of Section \ref{Franke}. For
model structure purposes we will usually assume that $\abcat$ is a
Grothendieck abelian category, which is our reason for choosing this
letter. We will always assume that we have a self-equivalence $T \co
\abcat \to \abcat$ and we further assume that $\abcat$ has all small
limits and colimits.

In this section, we introduce the category $C^{(T,N)}(\abcat)$ of
quasi-periodic chain complexes of period $N$, i.e. chain complexes
with values in $\abcat$ that are periodic
up to a ``twist'' by $T$. Given a model
structure on chain complexes $C(\abcat)$, we are then going to
discuss how $C^{(T,N)}(\abcat)$ inherits a model structure from
$C(\abcat)$.

\begin{definition}
The category, $C^{(T,N)}(\abcat)$, of \textbf{quasi-periodic chain complexes} (or
\textbf{twisted chain complexes}) in $\abcat$,
has objects the class of chain complexes $C(\abcat)$ in $\abcat$
which have a specified isomorphism $ \alpha_C: T(C_*)
\longrightarrow C[N]_* $. A morphism is then a chain map which
commutes with the given isomorphisms as above.
\end{definition}

Here, $C[N]_*=C_{*-N}$, the differential on $C[N]$ is
$d_{C[N]}=(-1)^{N}d_C$. For further details on this category see
\cite[Example 1.3.3]{Fra96} or \cite[Subsection 2.2]{Roi08}.
%Clearly, $C^{(\id,0)}(\abcat) = C(\abcat)$ is just the category
%of chain complexes of objects of $\abcat$.

The forgetful functor
$U: C^{(T,N)}(\abcat ) \longrightarrow C(\abcat )$
from quasi-periodic chain complexes to chain complexes on
$\abcat$ has both a left and a right adjoint. We are most
interested in the left adjoint, which we call
\textbf{periodification}.  Given a chain complex $M$, we define
$$\mathbb{P}M = \bigoplus\limits_{k \in\zz} T^k M[-kN].$$
Thus $\mathbb{P}M_n = \bigoplus\limits_{k \in\zz} T^k M_{n+kN}.$
The differential on summand $T^k M_{n+kN}$ is given by
$$
(-1)^{kN} T^k d_{n+kN} \co T^k M_{n+kN} \to T^k M_{n+kN-1}.
$$
This is a functor, the action on maps being to send $g$ to that map
which on level $n$ and summand $k$ is given by $T^k g_{n+kN}$. Furthermore $\mathbb{P}M$
is a quasi-periodic chain complex, the quasi-periodicity isomorphism is the following
composite.
$$
\begin{array}{rcl}
T\mathbb{P}M_n = T\bigoplus\limits_{k \in\zz} T^k M_{n+kN} & \cong &
\bigoplus\limits_{k \in\zz} T^{k+1} M_{n+kN}\\
& = & \bigoplus\limits_{k \in\zz} T^{k+1} M_{n+(k+1)N-N} \\
& = & \bigoplus\limits_{k \in\zz} T^{k+1} M[N]_{n+(k+1)N} \\
& = & \mathbb{P}M[N]_n \\
\end{array}
$$

\begin{lemma}
The functor $\mathbb{P}$ is the left adjoint to the forgetful
functor $U$ from quasi-periodic chains on $\abcat$ to chains on
$\abcat$.
\end{lemma}
\begin{proof}
Let $f \co \mathbb{P}M \to X$ be a quasi-periodic chain map. Let
$f^n_k$ be the map from the $k$-summand of $\mathbb{P}M_n$ to $X_n$,
so
$
f^n_k \co T^k M_{n+kN} \to X_n.
$
The collection $f^n_0 \co M_n \to X_n$ defines a chain map $\hat{f}
\co M \to X$.

For the inverse, let $g \co M \to X$ be a chain map.
Define a collection $g^n_k$ by the following composite, where the
second map is coming from the quasi-periodic structure of $X$.
$$
T^k M_{n+kN} \xrightarrow{T^k g_{n+kN}} T^k X_{n+kN}
\xrightarrow{(\alpha_X)^k} X_n
$$
We then define a map $\tilde{g} \co \mathbb{P}M \to X$, which on
summand $k$ is $g^n_k$. To see that $\tilde{g}$ is a quasi-periodic map, we
note that the following diagram commutes
\[
\xymatrix{ T\mathbb{P} M \ar[r]^{T\tilde{g}}\ar[d]^{\alpha_{\mathbb{P}M}} & TX \ar[d]^{\alpha_X} \\
\mathbb{P} M[N] \ar[r]^{\tilde{g}[N]} & X[N] }
\]
because the diagram below commutes.
$$
\xymatrix@R=20pt@C=50pt{ TT^k M_{n+kN} \ar[r]^{TT^k g_{n+kN}}
\ar[d]^= & TT^k X_{n+kN} \ar[r]^{(\alpha_X)^k} \ar[d]^= &
T X^{n}  \ar[d]^{\alpha_X} \\
T^{k+1} M_{n+kN} \ar[r]^{T^{k+1} g_{n+kN}} & T^{k+1} X_{n+kN}
\ar[r]^{(\alpha_X)^{k+1}} &
X_{n-N}  \\
}
$$
It is simple to check that $\tilde{g}$ is compatible with the
differentials. That $\hat{\tilde{g}}=g$ is immediate. One can also
check that $\tilde{\hat{f}}=f$, by noting that the diagram below
commutes. This holds since $f$ is quasi-periodic and the top path is
$\tilde{\hat{f}}_k^n$ whereas the lower path is $f^n_k$.
$$
\xymatrix{ T^k M_{n+kN} \ar[r]^{T^k f^{n-kN}_0} \ar[d]^= &
T^k X_{n+kN}  \ar[d]^{(\alpha_X)^{k}} \\
T^k M_{n+kN} \ar[r]^{f^{n}_k}  &
X^{n}  \\
}$$
\end{proof}

A  dual argument will show that the forgetful functor
has a right adjoint $\mathbb{R}$, which sends a chain complex $M$
to the quasi-periodic chain complex
$\mathbb{R}M = \prod_{k \in\zz} T^k M[-kN]$.

\begin{proposition}\label{create}
Assume that there is a cofibrantly generated model structure on $C(\abcat)$
and that $T$ is a left Quillen functor.
Then the forgetful functor
\[
U \co C^{(T,N)}(\abcat) \longrightarrow C(\abcat)
\]
creates a model structure on $C^{(T,N)}(\abcat)$. That is, there is a model structure on
the category of quasi-periodic chain complexes, $C^{(T,N)}(\abcat)$,
 where a map $f$ is a weak equivalence or a
fibration if and only if $Uf$ is so in $C(\abcat)$.
\end{proposition}

\begin{proof} Let $I$ be the generating cofibrations and $J$ the
generating trivial cofibrations of $C(\abcat)$. Using the lifting
criterion of \cite[Theorem 11.3.2]{hir03}, one only needs to show
that $\mathbb{P}I$ and $\mathbb{P}J$ satisfy the small-object
argument and that a $\mathbb{P}J$-cell complex is a weak
equivalence. Since the forgetful functor is both a left and a right
adjoint, it preserves all colimits, thus $\mathbb{P}I$ and
$\mathbb{P}J$ satisfy the small-object argument. Since $\mathbb{P}$
preserves acyclic cofibrations, it follows that applying $U$ to a
$\mathbb{P}J$-cell complex gives an acyclic cofibration in
$C(\abcat)$. Thus the forgetful functor takes $\mathbb{P}J$-cell
complexes to weak equivalences.
\end{proof}

The above proposition implies that the adjoint functor pair $(\mathbb{P},U)$ is a
Quillen adjunction with $\mathbb{P}$ being the left and $U$ being
the right Quillen functor.
Note that this model structure on $C^{(T,N)}(\abcat)$ is also
cofibrantly generated with the generating cofibrations being the
image of the generating cofibrations in $C(\abcat)$ under the
functor $\mathbb{P}$, see \cite[Appendix 1]{Sch07b}.

\begin{lemma}
The functor $\mathbb{P}$ preserves quasi-isomorphisms.
\end{lemma}
\begin{proof}
The functor $T$ is an equivalence, hence it preserves quasi-isomorphisms, as does
the shift functor. Since $\mathbb{P}$ is just an infinite direct sum of shifts and
applications of $T$, it preserves homology isomorphisms.
\end{proof}

One particular model structure that will be of interest
is the model structure created in \cite{Fra96}. It is
used on the algebraic model for the $K$-local stable homotopy
category. We will discuss this category in more detail in Section
\ref{Franke}. We add the assumption that $\abcat$ is a Grothendieck
abelian category so that we have cofibrant generation. Note that a
Grothendieck abelian category always has enough injectives
\cite[Corollary X.4.3]{Ste76} and every object is small
\cite[Proposition A.2.]{Hov01}.

\begin{definition}
A \textbf{Grothendieck abelian category} $\abcat$ is a cocomplete abelian category, where filtered colimits commute with finite limits. Further,
there is a \textbf{generator} $G$,  that is, $\abcat(G,-)$ is faithful.
\end{definition}

\begin{lemma}\label{injectivestructure}
If $\abcat$ is a Grothendieck abelian category, then there is a
cofibrantly generated model structure on $C^{(T,N)}(\abcat)$ with the
quasi-isomorphisms as the weak equivalences and the monomorphisms as the cofibrations.
We call this the \textbf{injective model structure}.
\end{lemma}

\begin{proof}
For $C(\abcat)$ this is \cite[Proposition 3.13]{Bek00}, which uses Theorem \ref{smith}.
We can lift this model structure to $C^{(T,N)}(\abcat)$ using the above results.
\end{proof}

At this point we would like to mention some parallels to topology. The main example of interest is the category $C^1(\mathcal{A})$, a special case of $C^{(T,N)}(\abcat)$. We introduce it in Section \ref{Franke}. It comes with an equivalence of triangulated categories
\[
\mathcal{R}: D^1(\mathcal{A}) \longrightarrow \Ho(L_1\mathcal{S}).
\]
Here, $D^1(\mathcal{A})$ is the homotopy category of $C^1(\mathcal{A})$ with the injective model structure. Further, $\Ho(L_1\mathcal{S})$ denotes the $K$-local stable homotopy category at $p>2$. The chain complex corresponding to the sphere $L_1 S^0$ under $\mathcal{R}$ is exactly $\mathbb{P}\mathcal{I}$. So in our construction, we would like the periodified unit $\mathbb{P}\mathcal{I}$ to play the role of a localised sphere.

The stable homotopy category itself is not equivalent to a category of chain complexes. However, the reader might find the constructions in this section similar to the following adjunction
\[
L_1 = -\wedge^L L_1 S^0:  \Ho(\mathcal{S}) \raisebox{-0.1\height}{$\overrightarrow{\longleftarrow}$} \Ho(L_1\mathcal{S}) :U.
\]
Localisation with respect to $K_{(p)}$ equals smashing with $L_1 S^0$, which is left adjoint to the forgetful functor.

\section{Quasi-periodic chain complexes as modules}\label{modules}

We describe the category of quasi-periodic chain
complexes as a category of modules over a specifically chosen
monoid. It gives a nice description of the monoidal
structure of $C^{(T,N)}(\abcat)$.

We now assume that $\abcat$ is a closed symmetric monoidal
category with tensor product $\otimes$ and unit $\mathcal{I}$. Here
$\abcat$ has both a tensor product and an internal homomorphism
object $F(-,-)$, which are related by the usual adjunction. In this case we
must make further assumptions on $T$. We want $T$ to behave like
$N$-fold suspension, in particular we do not require $T$ to be a
monoidal functor.

\begin{definition}\label{compatible}
We say that $T$ is \textbf{compatible} with the monoidal structure if
there is a natural isomorphism of functors
$$m \co T \to T\mathcal{I} \otimes (-).$$ 
\end{definition}
If $T$ is compatible with the monoidal structure then for any $X$ and $Y$ there is a natural isomorphism $T(X \otimes Y) \to TX \otimes Y$. 

\begin{lemma}\label{functioncompatible}
If $T$ is compatible with the monoidal structure, so is $T^n$, for any $n \in \zz$.
There are natural isomorphisms
$$
TF(X,Y) \cong F(T^{-1}X,Y) \cong F(X,TY).
$$
\end{lemma}

From now on we assume that $T$ is compatible with the monoidal structure. 
We can think of $\mathcal{I}$ as an object of $C(\abcat)$ concentrated in degree zero.
We show that $\mathbb{P}\mathcal{I}$ is a monoid in $C(\abcat)$ and that
the category of quasi-periodic chain complexes is isomorphic to
the category of $\mathbb{P}\mathcal{I}$-modules.

\begin{proposition}
The category of quasi-periodic chain complexes, $C^{(T,N)}(\abcat)$,
is isomorphic to the category of $\mathbb{P}\mathcal{I}$-modules in $C(\abcat)$.
\end{proposition}

%\begin{lemma}
%If $X$ is a quasi-periodic cochain complex, then $X$ has a natural action of $\mathbb{P}\mathcal{I}$.
%\end{lemma}

\begin{proof}
First of all, we prove that if $X$ is a quasi-periodic cochain
complex, then $X$ has a natural action of $\mathbb{P}\mathcal{I}$.
We start by writing out level $n$ of $\mathbb{P}\mathcal{I} \otimes X$,
the tensor product in the category of chain complexes:
$$
(\mathbb{P}\mathcal{I} \otimes X)_n = \bigoplus_{k \in \zz } T^k \mathcal{I}[-kN] \otimes X_{n+kN}
\cong \bigoplus_{k \in \zz } T^k X_{n+kN} \cong \mathbb{P}X_n
$$
the action map is then induced by the structure map of $X$ and the fold map
$$
\bigoplus_{k \in \zz } T^k X_{n+kN} \to \bigoplus_{k \in \zz } X_n \to X_n.
$$
It is easy to check that this map is associative and unital, the
unit of $\mathbb{P}\mathcal{I}$ being the obvious inclusion
$\mathcal{I} \to \mathbb{P}\mathcal{I}$. We note that the
differential of $$(\mathbb{P}\mathcal{I} \otimes X)_n \cong
\bigoplus_{k \in \zz } T^k X_{n+kN}$$ is given by $(-1)^{kN} T^k
d_{n+kN}$, which tells us that the differentials are compatible with
the action $\mathbb{P}\mathcal{I} \otimes X \to X$.

For the converse we prove that if $Y$ is a
$\mathbb{P}\mathcal{I}$-module, then $Y$ is a quasi-periodic chain
complex.
The action map of $Y$ takes the following form:
$$
(\mathbb{P}\mathcal{I} \otimes X)_n
\cong \bigoplus_{k \in \zz } T^k Y_{n+kN}
\to Y_n.
$$
Let $$\phi(k)_n \co T^k Y_{n+kN} \to Y_n$$ be the $k$ component of the above map.
We can assemble these to obtain a map $$\phi(k) \co T^k Y \to Y[kN].$$ We need to see that $\phi(1)$ is an isomorphism, it will be our periodicity map. Since the action map is unital, $\phi(0)$ is the identity. Associativity of the action shows that $$\phi(l+k) = \phi(k)\phi(l),$$ in particular $$\phi(1) \phi(-1) = \phi(0) = \phi(-1) \phi(1),$$ so $\phi(1)$ is an isomorphism.
\end{proof}

Another consequence of our computations is the following, which says that $\mathbb{P}$ is compatible with the monoidal structure.

\begin{corollary}
There is an isomorphism of quasi-periodic chain complexes
$\mathbb{P}\mathcal{I} \otimes X \cong \mathbb{P}X$, which is
natural in $X$.
\end{corollary}

Thus we have shown that the category of quasi-periodic chain complexes and
$\mathbb{P}\mathcal{I}$-modules are isomorphic. We can also think of
$\mathbb{P}$ as an monad on the category of
chain complexes of objects of $\abcat$, we can then
describe $C^{(T,N)}(\abcat)$ as the category of modules over this monad.
We make no use of this monad description.

A result by Schwede and Shipley states that if a monoidal model category satisfies the \textbf{monoid axiom} \cite[Definition 3.3]{SchShi00}, then there is an induced model structure on the category of modules over a fixed commutative monoid. We apply this result to our case and arrive at the following.

\begin{proposition}\label{modulemodel}
Assume that there is a model structure on $C(\abcat)$ which is cofibrantly generated, monoidal and satisfies the monoid axiom. Then the category of $\mathbb{P}\mathcal{I}$-modules has a cofibrantly generated model structure where the weak equivalences and fibrations are
the underlying weak equivalences and fibrations.
The generating cofibrations and acyclic cofibrations are given by
applying $\mathbb{P}\mathcal{I} \otimes (-)$ to the generating cofibrations and acyclic cofibrations
of $C(\abcat)$.
\end{proposition}

\begin{corollary}\label{modulemodelcor}
This model category on $C^{(T,N)}(\abcat)$ is precisely the model category of quasi-periodic chain complexes with the lifted model structure. Furthermore, since $\mathbb{P}\mathcal{I}$ is a commutative monoid, this model category is monoidal and satisfies the monoid axiom, with monoidal product given as the tensor over $\mathbb{P}\mathcal{I}$: $X \otimes_{\mathbb{P}\mathcal{I}} Y$.
\end{corollary}
This follows from \cite[Theorem 4.1]{SchShi00}.

\section{Comodules over Hopf algebroids}\label{hopfalgebroids}

We are going to recall some definitions, conventions and basic
properties about comodules over a Hopf algebroid as this is the abelian category we are most interested in. We refer to
\cite{Hov04} and \cite[Appendix B.3]{Rav92} for more details.

Let $k$ be a commutative ring. Then a \textbf{Hopf algebroid} is a
pair $(A,\Gamma)$ of commutative $k$-algebras such that for every
$k$-algebra $B$, the pair
\[
(\hom_{k-alg}(A,B),\hom_{k-alg}(\Gamma,B))
\]
forms a groupoid (i.e. a small category where every morphism is an
isomorphism) with $\hom_{k-alg}(A,B)$ being the objects and
$\hom_{k-alg}(\Gamma,B)$ being the morphisms. This means that there
are structure maps
\begin{itemize}
\item $\Delta: \Gamma \longrightarrow \Gamma \otimes_A \Gamma$
(coproduct, inducing composition of morphisms)
\item $c: \Gamma \longrightarrow \Gamma$ (conjugation, inducing
inverses)
\item $\epsilon: \Gamma \longrightarrow A$ (augmentation, inducing
identity morphisms)
\item $\eta_R: A \longrightarrow \Gamma$ (right unit, inducing
target)
\item $\eta_L: A \longrightarrow \Gamma$ (left unit, inducing
source)
\end{itemize}
satisfying certain conditions. Note that $\eta_R$ and $\eta_L$ make
$\Gamma$ into an $A$-bimodule. By $\otimes_A$ we mean the tensor product
of $A$-bimodules.

For technical reasons we are going to consider only \textbf{flat
Adams Hopf algebroids}, i.e. Hopf algebroids $(A,\Gamma)$ where
$\Gamma$ is a filtered colimit of finitely generated projectives
over $A$. The main topological example we have in mind are Hopf
algebroids of the form $(R_*, R_*R)$ where $R$ is a topologically
flat commutative ring spectrum.

\begin{definition}
A \textbf{$(A,\Gamma)$-comodule} is a left $A$-module $M$ together
with a map
\[
\psi_M: M \longrightarrow \Gamma \otimes_A M
\]
satisfying a coassociativity and counit condition.
\end{definition}
The category
$(A,\Gamma)\mbox{-comod}$ is a cocomplete abelian category
\cite[Lemma 1.1.1]{Hov04}.
It is also a closed monoidal category with symmetric monoidal
product $\wedge$ with unit $A$. For two $(A,\Gamma)$-comodules
$M$ and $N$, $M\wedge N$ denotes the tensor product of $M$
and $N$ as left $A$-modules ($A$ is
assumed to be commutative). The comodule structure map is then given
by
\[
M \otimes N \xrightarrow{\psi_M \otimes \psi_N} (\Gamma \otimes_A M)
\otimes (\Gamma \otimes_A N) \xrightarrow{\gamma} \Gamma \otimes_A
(M \otimes N)
\]
where $$\gamma((x \otimes m) \otimes_A (y \otimes n)) = xy \otimes_A
(m \otimes n),$$ see \cite[Lemma 1.1.2]{Hov04}. As for the closed
structure, the right adjoint of the monoidal product $\wedge$ is
denoted by $F(-,-)$. It is left exact in the first variable and
right exact in the second one. If $M$ is finitely presented over
$A$, then $$F(M,N) \cong {A\mbox{-mod}}(M,N)$$ as $A$-modules, but
$F(M,N)$ is generally not isomorphic to $(A,\Gamma)$-comod$(M,N)$.
For more properties of $F$ see \cite[Subsection 1.3]{Hov04}.

When one has a flat Adams Hopf algebroid, the category of comodules
is a Grothendieck abelian category by \cite[Propositions 1.4.1 and 1.4.4]{Hov04}.
This property will be important for Sections \ref{relativeprojective} and
\ref{quasiprojective}.

By $C(A,\Gamma)$ we mean the category of chain complexes in
$(A,\Gamma)$-comodules. There are two model structures on
$C(A,\Gamma)$ as described in \cite{Hov04}, namely the relative
projective model which we consider in Section \ref{relativeprojective}
and the homotopy model structure. The homotopy model
structure is a Bousfield localisation of the relative model
structure and has various technical advantages over the latter.
However, for our purposes it is more appropriate to consider the relative model
structure. Let us summarise a few properties.

\begin{theorem}[Hovey]
The relative projective model structure on $C(A,\Gamma)$ for a flat Adams Hopf algebroid
$(A,\Gamma)$ is cofibrantly (and finitely) generated, proper, stable and monoidal.
The cofibrations are precisely the degreewise split
monomorphisms whose cokernel is a complex of relative projectives
with no differential.
It satisfies the monoid axiom and if $X$ is cofibrant and $f$ is a projective
equivalence, then $X \wedge f$ is a projective equivalence.
\end{theorem}

\section{Franke's exotic model}\label{Franke}

For a spectrum $E$, the $E$-local stable homotopy category is
obtained from the stable homotopy category by formally inverting
those maps that induce isomorphisms in $E_*$-homology. The
resulting category is especially sensitive towards phenomena related
to $E_*$. For certain special homology theories this is an
important structural tool for studying the stable homotopy category
itself.

Jens Franke used quasi-periodic chain complexes to give an algebraic
description of $\Ho(L_1 \mathcal{S})$, the $K$-local stable homotopy
category at an odd prime $p$. Equivalently, one can consider the $E(1)$-local
stable homotopy category for the $p$-local Adams summand $E(1)$.
We briefly recall Franke's result, the abelian
categories and the self-equivalences used in it.

\begin{theorem}[Franke]
There is an equivalence of categories
\[
\mathcal{R}: D^{2p-2}(\mathcal{B}) \longrightarrow \Ho(L_1
\mathcal{S})
\]
where $D^{2p-2}(\mathcal{B})$ denotes the derived category of
quasi-periodic chain complexes over the abelian category
$\mathcal{B}$ and $\Ho(L_1 \mathcal{S})$ the $E(1)$-local stable
homotopy category. Further, there is a natural isomorphism
\[ E(1)_*(\mathcal{R}(C)) \cong
\bigoplus\limits_{i=0}^{2p-3} H_i(C)[i].\]
\end{theorem}

We would like to remark that Franke's theorem also holds for $\Ho(L_n \mathcal{S})$ whenever $n^2 + n <2p-2$. However, the description of the abelian category is less explicit. This is why we only formulate it for $n=1$ and $p>2$, although our main results will also hold in the whole of Franke's range.

\medskip
Let us recall the ingredients of this theorem. We will first
describe a category $\mathcal{A}$ which is equivalent to the
category of $E(1)_*E(1)$-comodules as introduced by Bousfield in
\cite{Bou85}, see also \cite{ClaCroWhi07}.

\begin{definition}
Let $p$ be an odd prime, set $\mathcal{B}$ to be the category of modules over $\zz_{(p)}$ (the $p$-local integers), with Adams operations $\psi^k$, $k \in \zz_{(p)}^*$, such that, for each $M \in \mathcal{B}$:
\begin{itemize}
\item There is an eigenspace decomposition
$$M \otimes \qq \cong \bigoplus\limits_{j \in \zz} W_{j(p-1)}$$
such that for all $w \in W_{j(p-1)}$, and $k \in \zz_{(p)}^*$, $(\psi^k \otimes \id)w= k^{j(p-1)}w$.

\item For each $x \in M$ there is a finitely generated submodule $C(x)$, which contains $x$, such that for all $m \geqslant 1$, there is an $n$ such that the action of $\zz_{(p)}^*$ on $C(x)/p^m(x)$ factors through the quotient of $(\zz/p^{n+1})^*$ by its subgroup of order $p-1$.
\end{itemize}
\end{definition}

There is a self-equivalence $T^{j(p-1)} \co \mathcal{B} \to
\mathcal{B}$, for each $j \in \zz$. It leaves the underlying
$\zz_{(p)}$-module unchanged but $\psi^k$ acts on this as
$k^{j(p-1)}\psi^k$ ($k \in \zz_{(p)}^*$).

\begin{definition}
The objects of the category $\mathcal{A}$ are collections $(M_n)_{n
\in \zz}$, with $M_n \in \mathcal{B}$, with specified isomorphisms
$T^{p-1}(M_n) \to M_{n+2p-2}$, for each $n \in \zz$.
\end{definition}

The category $\mathcal{B}$ is then a subcategory of $\mathcal{A}$,
an object $M \in \mathcal{B}$ can be viewed as a collection
$(M_n)$, where $M_n=M$ whenever $n \equiv 0 \mod 2p-2$ and is zero
elsewhere. The isomorphisms are then the identity on objects. So the
category $\mathcal{A}$ is isomorphic to the sum of $2p-2$ shifted
copies of $\mathcal{B}$.

\begin{theorem}[Bousfield, Clarke-Crossley-Whitehouse]
The abelian category $\mathcal{A}$ is isomorphic to the category of
$(E(1)_*,E(1)_*E(1))$-comodules.
\end{theorem}

We note here that the Hopf algebroid
$(E(1)_*,E(1)_*E(1))$ is a flat Adams Hopf algebroid \cite[Theorem
1.4.9]{Hov04}. Hence we are consistent with the technical
assumptions needed for talking about the relative projective model
structure on $C^{(T,N)}(\mathcal{A})$ later.

\begin{definition}
We define $C^1(\mathcal{A})$ to be $C^{(T^{p-1},1)}(\mathcal{A})$.
Similarly, we rename the category
$C^{(T^{(2p-2)(p-1)},2p-2)}(\mathcal{B})$ as
$C^{2p-2}(\mathcal{B})$. We also rename both
$T^{p-1}$ and $T^{(2p-2)(p-1)}$ as $T$.
Note that the two categories are isomorphic.
\end{definition}

\begin{lemma}
The self-equivalence $T$ is compatible with the
monoidal structure on
$\mathcal{A}$ in the sense of Definition \ref{compatible}.
\end{lemma}

\begin{proof}
The monoidal product on $\mathcal{B}$ is given by
tensoring over $\zz_{(p)}$ and allowing $\zz_{(p)}^*$
to act diagonally.
So on $X \otimes_{\zz_{(p)}} Y$, $$\psi^k(a \otimes b) =
\psi^k a \otimes_{\zz_{(p)}} \psi^k  b.$$

That this product structure satisfies the compatibility conditions is easy to check.
The natural isomorphism $$T(X \otimes_{\zz_{(p)}} Y) \to TX
\otimes_{\zz_{(p)}} Y$$ is the identity map on underlying sets. To
see that this morphism is a map of $\mathcal{B}$, we note that
$T(X \otimes_{\zz_{(p)}} Y)$, $\psi^k$ acts as $k^{p-1} (\psi^k
\otimes \psi^k)$, whereas on $TX \otimes_{\zz_{(p)}} Y$, $\psi^k$
acts as $(k^{p-1} \psi^k) \otimes \psi^k$. Since we have tensored
over $\zz_{(p)}$, these are the same.
We extend this to a monoidal
product on $\mathcal{A}$ in the standard manner:
$$(M \smashprod N)_n = \bigoplus_{a+b=n} M_a \otimes_{\zz_{(p)}} N_b.$$
The unit is best described as $E(1)_*$.
It follows immediately that this tensor product is compatible with
$T$. One could see this equally well by considering the monoidal
structure on $E(1)_*E(1)$-comodules, see \cite{ClaCroWhi07}.
\end{proof}

Following Section \ref{modules}, the monoidal product defined in the above proof
induces a closed monoidal structure on
the category $C^1(\mathcal{A})$.

Franke constructs a model structure on quasi-periodic chain complexes
as follows, see also \cite[Example 1.3.3]{Fra96}. A quasi-periodic chain map
$f:X \longrightarrow Y$ is:
\begin{itemize}
\item a weak equivalence if it is a quasi-isomorphism
\item a fibration if it is a degreewise split epimorphism with
strictly injective kernel
\item a cofibration if it is a monomorphism.
\end{itemize}

We call this the \textbf{injective model structure}, we briefly mentioned
this model structure in Lemma \ref{injectivestructure}. A
quasi-periodic chain complex is $C$ is said to be \textbf{strictly injective} if it is
degreewise injective and every morphism from $C$ into an acyclic
complex $K$ is nullhomotopic via a quasi-periodic homotopy.

\begin{definition}
In the above special case we denote the respective derived
categories of $C^1(\mathcal{A})$ and $C^{2p-2}(\mathcal{B})$ by
$D^1(\mathcal{A})$ and $D^{2p-2}(\mathcal{B})$.
\end{definition}

The main defect of the injective model
structure is that it is not monoidal (the pushout-product axiom fails).
The counterexample is analogous to the one in the injective model structure on chain complexes of $R$-modules
for a ring $R$ \cite[Subsection 4.2]{Hov99}. The pushout-product axiom states (in part)
that in a monoidal model category with product $\otimes$, if one takes two
cofibrations $f:U \longrightarrow V$ and $g:W \longrightarrow X$, then the
induced map
\begin{equation*}\label{box}
f \square g: (V\otimes W) \coprod_{U \otimes W} (U \otimes X)
\longrightarrow V \otimes X
\end{equation*}
is again a cofibration. To see that
this is not the case for $C^{1}(\mathcal{A})$ with the injective
model structure we take 
\[
U= \mathbb{P}\mathcal{I},\,\,\, V=\mathbb{P}(\mathcal{I}
\otimes_{\mathbb{Z}_{(p)}} \mathbb{Q}),\,\,\, W=0 \,\,\,\,\mbox{and}\,\,\,\, X=\mathbb{P}(\mathcal{I}
\otimes_{\mathbb{Z}_{(p)}} \mathbb{Z}/p)
\]
with $f$ and $g$ being the
obvious inclusions. Remembering that
\[
\mathbb{P}C \smashprod_{\mathbb{P}\mathcal{I}} \mathbb{P}D = \mathbb{P} (C \smashprod_\mathcal{I} D),
\]
we see that the induced
pushout-product map is $X \longrightarrow 0$, which is clearly not a
monomorphism and hence not a cofibration.

\section{The relative projective model
structure}\label{relativeprojective}

In this section we are going to summarise the relative projective
model structure on $C(\abcat)$. It is a generalisation of the
projective model structure on $C(R\mbox{-mod})$ where $R$ is a
commutative ring. It was introduced by Christensen and Hovey in
\cite{ChrHov02}. Assuming that the relative projective model
structure exists on $C(\abcat)$ for some $\abcat$, we are
going to discuss the model structure it creates on the
quasi-periodic chains $C^{(T,N)}(\abcat)$. At the end of this section
we specialise to the case of $C^{1}(\mathcal{A})$.

One begins by specifying the objects playing the role of the
``projective'' objects. This class of chosen objects is called a \textbf{projective class}
$\mathcal{P}$, see \cite[Def. 1.1]{ChrHov02}. The objects $P \in
\mathcal{P}$ are called \textbf{relative projectives}. A morphism
$f:A \longrightarrow B$ in $\abcat$ is called
\textbf{$\mathcal{P}$-epimorphism} if it induces an epimorphism
${\abcat}(P,A) \longrightarrow {\abcat}(P,B)$
for all $P$ in $\mathcal{P}$.
Assuming that $\abcat$ is cocomplete, one way to obtain a
projective class is to take any set $\mathcal{S}$ and define
$\mathcal{P}$ to be the collection of retracts of coproducts of
objects in $\mathcal{S}$, \cite[Lemma 1.5]{ChrHov02}.

We use the projective class to define a model structure on $C(\abcat)$.
We say that a chain map $f:X \longrightarrow Y$ is:
\begin{itemize}
\item a \textbf{$\mathcal{P}$-equivalence} if $f_*: {\abcat}(P,X) \longrightarrow
{\abcat}(P,Y)$ is a quasi-isomorphism in $C(\mathbb{Z})$
for all $P \in \mathcal{P}$  (note that ${\abcat}(P,X)$ is
a chain complex in the usual way with differential
$(d_X)_* \co {\abcat}(P,X_n)\longrightarrow {\abcat}(P,X_{n-1})$),
\item a \textbf{$\mathcal{P}$-fibration} if ${\abcat}(P,f)$
is a degreewise surjection for all $P \in \mathcal{P}$,
\item a \textbf{$\mathcal{P}$-cofibration} if it has the left lifting
property with respect to all $\mathcal{P}$-fibrations that are also
$\mathcal{P}$-equivalences.
\end{itemize}

\begin{theorem}[Christensen-Hovey]\label{relativemodelstructure}
The above three classes form a
model structure on the category of chain complexes $C(\abcat)$ if and only if cofibrant
replacements exist.
\end{theorem}

This model structure is called the \textbf{relative projective model
structure}, it is proper whenever it exists.
Christensen and Hovey also characterise the cofibrant
objects in this model structure.

\begin{proposition}
A chain map $i:A \longrightarrow B$ is a $\mathcal{P}$-cofibration
in $C(\abcat)$ if and only if it is a degreewise split
monomorphism with $\mathcal{P}$-cofibrant cokernel.

A chain complex $C$ is cofibrant in $C(\abcat)$ if and only if
it is degreewise relative projective and every map from $C$ to a
weakly $\mathcal{P}$-contractible chain complex $K$ is
nullhomotopic.
\end{proposition}

Here, \textbf{weakly $\mathcal{P}$-contractible} means
$\mathcal{P}$-equivalent to 0.

\begin{theorem}[Hovey]\label{Grelmodel}
If $\abcat$ is a Grothendieck abelian category and the projective class
is constructed from a set $\mathcal{S}$ (using retracts and coproducts), then the
relative projective model structure on $C(\abcat)$ exists
and is cofibrantly generated.
\end{theorem}

The statement regarding cofibrant generation is \cite[Theorem 5.7]{ChrHov02},
the generating sets are below.
$$
I=\{ S^{n-1} P \to D^n P | P \in \mathcal{S}, \
n \in \mathbb{Z}  \} \quad J=\{ 0 \to D^n P | P \in
\mathcal{S}, \ n \in \mathbb{Z}  \}
$$
As usual, $S^{n-1}P$ denotes the chain complex that only consists of
$P$ concentrated in degree $n-1$ and $D^n P$ is the chain complex
with $P$ in degrees $n-1$ and $n$ (and zeroes elsewhere) with the
identity as the only non-trivial differential.

We further assume that $T(\mathcal{P})=\mathcal{P}$. This implies that $T$ is
left Quillen functor, as it preserves the generating sets above.

\begin{proposition}\label{randomlabel}
Say that a map of $C^{(T,N)}(\abcat)$ is
a \textbf{$\mathcal{P}$-equivalence} or a
\textbf{$\mathcal{P}$-fibration} if it is so as a map of $C(\abcat)$.
Then these classes of maps define a
cofibrantly generated model structure on quasi-periodic chains
$C^{(T,N)}(\abcat)$.
\end{proposition}

Before we prove the proposition, it is worth mentioning that quasi-isomorphisms are not necessarily $\mathcal{P}$-equivalences. But in Franke's model, the weak equivalences are exactly the quasi-isomorphisms. We will see in Proposition \ref{quillenstep} that consequently the relative projective model structure and Franke's injective model structure are not Quillen equivalent. However, the relative projective model structure provides a vital intermediate step towards a model structure Quillen equivalent to Franke's model with better properties than the injective model.

We continue with the proof of Proposition \ref{randomlabel}
\medskip
\begin{proof}
It is immediate that this model structure is precisely that created by the
forgetful functor
$$
U: C^{(T,N)}(\abcat) \longrightarrow C(\abcat)
$$
as discussed in Proposition \ref{create}.
The generating cofibrations and acyclic cofibrations are given by the sets
$$
\mathbb{P}I=\{ \mathbb{P}S^{n-1} P \to \mathbb{P}D^n P | P \in \mathcal{S}, \
n \in \mathbb{Z}  \} \quad \mathbb{P}J=\{ 0 \to \mathbb{P}D^n P | P \in
\mathcal{S}, \ n \in \mathbb{Z}  \}
$$
where $\mathbb{P}$ is the periodification functor defined in Section
\ref{adjoints}.
\end{proof}

\begin{corollary}
A cofibration of $C^{(T,N)}(\abcat)$ is a cofibration of
$C(\abcat)$. An acyclic cofibration of $C^{(T,N)}(\abcat)$ is an
acyclic cofibration of $C(\abcat)$. Thus the functor $\mathbb{R}$
defined in Section \ref{adjoints} is a right Quillen functor, with
left adjoint $U$.
\end{corollary}
\begin{proof}
We prove the first of these statements, the proof of the second is
identical (it also follows from the fact that $U$ preserves cofibrations
and weak equivalences). The third follows immediately.

Since $T$ is a left Quillen functor, the periodification $\mathbb{P}
\co C(\abcat) \to C(\abcat)$ is also a left Quillen functor. Hence
the set $\mathbb{P} I$ above consists of cofibrations of $C(\abcat)$. It
follows immediately that $\mathbb{P} I \cof$ (as constructed in the category
$C(\abcat)$) is contained in the class of cofibrations of
$C(\abcat)$. In turn, applying $U$ to an element of the class $\mathbb{P} I \cof$
(as constructed in the category $C^1(\abcat)$) gives a cofibration of
$C(\abcat)$.
\end{proof}

Let us now give a characterisation of the cofibrant objects in
$C^{(T,N)}(\abcat)$ and the cofibrations. These results follow
a well-known standard argument (for an example see \cite[Section
2]{ChrHov02}) but we include them for completeness' sake.
Of course, since $C^{(T,N)}(\abcat)$ is cofibrantly generated, we can use the usual description
of cofibrant objects as retracts of relative cell complexes.

\begin{lemma}\label{cofibrant}
A quasi-periodic chain complex $C$ is cofibrant in
$C^{(T,N)}(\abcat)$ if and only if it is degreewise relative
projective and every map from $C$ to a weakly
$\mathcal{P}$-contractible quasi-periodic chain complex $K$ is
nullhomotopic with quasi-periodic homotopy.
\end{lemma}

\begin{proof}
Let $C$ be degreewise relative projective and assume that every map
from $C$ to a weakly $\mathcal{P}$-contractible quasi-periodic chain
complex $K$ is nullhomotopic with quasi-periodic homotopy. We are
going to show that the inclusion $0 \longrightarrow C$ has the left
lifting property with respect to all acyclic fibrations $f: X
\longrightarrow Y$, that is, there is a lift in the diagram
\[
\xymatrix{ 0 \ar[r]\ar[d] & X \ar@{->>}[d]_{f}^{\sim} \\
C \ar[r]_{g}\ar@{.>}[ur]^{\tilde{g}}& Y. }
\]
As $C$ is degreewise relative projective and $f: X \longrightarrow
Y$ is a $\mathcal{P}$-epimorphism, there are degreewise lifts
$\gamma_n: C_n \longrightarrow X_n$.
We can choose those lifts to be quasi-periodic, so
$$\gamma_{n-N}=\alpha_X \circ T\gamma_n \circ \alpha_C^{-1},$$
by simply choosing lifts in degrees 0 to $N-1$ and extending.
However, these maps $\gamma_n$ do not necessarily form a
chain map, so we are going to add an extra term to obtain a chain
map.

Consider the degree-wise defined map $$\partial:=d_X\circ \gamma -
\gamma \circ d_C$$ from $C$ to $X$.  Then $f \circ \partial=0$, so there is a lift
$F:C \longrightarrow K[1]$ where $K$ is the kernel of the acyclic
fibration $f$, thus $F\circ j = \partial$, where $j \co K
\longrightarrow X$ is the inclusion. One can check that $F$ is not
just a degreewise map in $\abcat$ but a chain map. (Note that
$d_{K[1]}=-d_K$.) This map $F$ can also be chosen to be a
quasi-periodic map.

The kernel $K$ is weakly contractible, so by assumption $F$ is
nullhomotopic with quasi-periodic nullhomotopy, i.e. there is a
family of maps
$$h_n:C_n \longrightarrow K_n $$ such that
$$F_n=h_{n-1}\circ d_C + d_{K[1]}\circ h_n$$ and \quad
$$h_{n-N}=\alpha_K \circ Th_n \circ \alpha_C^{-1}.
$$
Now define the desired lift $\tilde{g}$ as $\tilde{g}:= \gamma +
j\circ h$. This is a quasi-periodic chain map by construction and
satisfies $f \circ \tilde{g}=g$.

%Conversely, let $C \in C^{(T,N)}(\abcat)$ be cofibrant. Let $f: M \longrightarrow N$ be a $\mathcal{P}$-epimorphisms. Further, let $\mathbb{R}$ denote the right adjoint to the forgetful functor as introduced in Section \ref{adjoints}. Then the induced map between $n$-discs $\mathbb{R}D^n(M) \longrightarrow \mathbb{R}D^n(N)$ is an acyclic $\mathcal{P}$-fibration. Hence the map
%\[
%\hom_{C^{(T,N)}(\abcat)}(C, \mathbb{R}D^n(M)) \longrightarrow \hom_{C^{(T,N)}(\abcat)}(C,\mathbb{R}D^n(N) )
%\]
%is a surjection because $C$ is assumed to be cofibrant. By adjunction, this is equivalent to saying that
%\[
%\hom_{\abcat}(C_n,M) \longrightarrow \hom_{\abcat}(C_n,N)
%\]
%is a surjection for every $\mathcal{P}$-epimorphism $f: M \longrightarrow N$, hence $C_n \in \mathcal{P}$.

Conversely, let $C \in C^{(T,N)}(\abcat)$ be cofibrant. Because $C$
is also cofibrant in $C(\abcat)$ by Corollary \ref{cofibrant}, we
know that $C$ is degreewise relative projective. Further, let $f:C
\longrightarrow K$ be a morphism with $K \in C^1(\abcat)$ be weakly
contractible. Consider the quasi-periodic chain complex $PK:= K
\oplus K[-1]$ with differential $d(x,y)=(dx,x-dy)$. The projection
$p: PK \longrightarrow K$ is an acyclic fibration, so $f$ factors
over $p$ because $C$ is cofibrant.
So there is a lift

\[
\xymatrix{ 0 \ar[r]\ar[d] & PK\ar@{->>}[d]_{p}^{\sim} \\
C \ar[r]_{f}\ar@{.>}[ur]^{\tilde{f}}& K. }
\]

with $\tilde{f}=(f,h)$ where $h: C \longrightarrow K[-1]$ is a quasi-periodic chain map. Because $\tilde{f}$ is also a quasi-periodic chain map, we have
$$
f= d\circ h + h \circ d,$$ so $h$ also serves as a quasi-periodic nullhomotopy of $f$, which is what we wanted to prove.
\end{proof}

\begin{lemma}\label{periodiccofib}
A quasi-periodic chain map $i:A \longrightarrow B$ is a
$\mathcal{P}$-cofibration in $C^{(T,N)}(\abcat)$ if and only if
it is a degreewise split monomorphism with $\mathcal{P}$-cofibrant
cokernel.
\end{lemma}

\begin{proof}
The cokernel of a cofibration is cofibrant as it is the pushout of $i$ along the zero map, and cofibrations are invariant under pushouts. It is a split monomorphism, because by Corollary \ref{cofibrant}, it is a cofibration in $C(\abcat)$.

Now let $i:A \longrightarrow B$ be a degreewise split monomorphism with cofibrant cokernel $C$. We would like to show that $i$ has the LLP with respect to an acyclic fibration $p: X \longrightarrow Y$ as in the following diagram
\[
\xymatrix{
A \ar[r]^{f}\ar[d]^{i} & X \ar@{->>}[d]_{p}^{\sim} \\
B \ar[r]_{g}\ar@{.>}[ur]^{\tilde{g}}& Y.
}
\]
Because $i$ is a split monomorphism, we can write $B=A \oplus C$ and $g=(g_A,g_C)$. Since $C$ is cofibrant, there is a lift $\tilde{g}_C$ similarly to the previous lemma. Hence $\tilde{g}:=(f,\tilde{g}_C)$ is the desired lift in the diagram, so $i$ is a cofibration.
\end{proof}

Using the results of Section \ref{modules} we can consider monoidal structures on
$C^{(T,N)}(\abcat)$, Corollary \ref{modulemodelcor}
gives the following result.

\begin{proposition}\label{relprojectivemonoidal}
Assume that the relative projective model structure on $C(\abcat)$
is a monoidal model category that satisfies the monoid axiom. Assume
further that $T(\mathcal{P})=\mathcal{P}$ and that $T$ is compatible with the
monoidal structure of $\abcat$ in the sense of Definition
\ref{compatible}. Then the induced model structure on
$C^{(T,N)}(\abcat)$ is monoidal and satisfies the monoid axiom.
\end{proposition}

We now turn to our motivating example: the category of
$(A,\Gamma)$-comodules for $(A,\Gamma)$ a flat Adams Hopf algebroid,
see Section \ref{hopfalgebroids}. We assume that this category has a
self-equivalence $T$ that is compatible with the monoidal product.
We introduce a particular projective class, we will use it to
construct the quasi-projective model structure and see that it has some
useful properties.

Remember that $F$ denotes the function object of a closed monoidal
category and $\mathcal{I}$ denotes the unit.

\begin{definition}
The \textbf{dual} of an $(A,\Gamma)$-comodule $M$ is
$DM=F(M,\mathcal{I})$. A comodule is
called \textbf{dualisable} if the natural
map $DM \wedge N \longrightarrow F(M,N)$ is an isomorphism for all
$N$.
\end{definition}

In the case of flat Adams Hopf algebroid, a comodule is dualisable
if and only if it is finitely generated and projective as an underlying
$A$-module \cite[Proposition 1.3.4]{Hov04}. As a consequence, the collection of
isomorphism classes of dualisable comodules is a set $\mathcal{S}$.
The projective class associated to this set \cite[Lemma
1.5]{ChrHov02} gives the the relative projective model structure on
chain complexes of comodules via Theorem \ref{Grelmodel}. We now
exploit the good monoidal properties obtained by choosing the
projective class to come from the dualisable objects. In particular,
we now longer have to worry about asking for
$T(\mathcal{P})=\mathcal{P}$.

\begin{lemma}
If $T$ is compatible with the monoidal structure on the category of $(A,\Gamma)$-comodules,
then $P$ is dualisable if and only if $T^n P$ is dualisable for each $n \in \zz$.
\end{lemma}
\begin{proof}
We have the following commutative diagram (see Lemma
\ref{functioncompatible}), from which the result follows.
$$
\xymatrix{
F(TP,X) \ar[r]^\cong \ar[d] & T^{-1} F(P,X) \ar[dd]^\cong  \\
F(TP,\mathcal{I}) \smashprod X \ar[d]^\cong \\
T^{-1} F(P,\mathcal{I}) \smashprod X \ar[r]^\cong &
T^{-1} (F(P,\mathcal{I}) \smashprod X)
}
$$
\end{proof}

\begin{lemma}
The monoidal product of two dualisable comodules is dualisable.
The dual of a dualisable comodule is dualisable.
There is a natural map $X \to DDX$ that is an isomorphism if $X$ is dualisable.
\end{lemma}
All of these results on dualisable objects hold in a general closed monoidal category,
see \cite[Chapter III]{LewMaySte86} for a detailed description.
Together with Proposition \ref{relprojectivemonoidal} and
\cite[Proposition 2.1.4]{Hov04} we obtain the following corollary.

\begin{corollary}
In the context of Section \ref{Franke}, $C^1(\mathcal{A})$ with the
relative projective model structure is monoidal and satisfies the
monoid axiom.
\end{corollary}

The monoidal product of two objects $M$ and $N$ of
$C^{(T,N)}((A,\Gamma)\mbox{-comod})$ is given by  $M
\smashprod_{\mathbb{P} \mathcal{I}} N$. This product is particularly well
behaved,  as well as satisfying the pushout product axiom, we have
the lemma below, which will make it easier to calculate the derived
monoidal product on $D^{(T,N)}(\abcat)$.

\begin{lemma}
If $X$ is a $\mathcal{P}$-cofibrant quasi-periodic chain complex,
then the functor
$X \smashprod_{\mathbb{P} \mathcal{I}} (-)$ preserves $\mathcal{P}$-equivalences. Furthermore
every $\mathcal{P}$-cofibrant object $X$ of $C^1(\abcat)$ is a
retract of some $\mathbb{P}Y$, where $Y$ is an $I$-cell complex of
$C((A,\Gamma)\mbox{-comod})$.
\end{lemma}

\begin{proof}
The quasi-periodic chain complex $X$ is a retract of some
$\mathbb{P}I$-cell complex $Z$. The quasi-periodic chain complex $Z$
is constructed as a colimit of pushouts of coproducts of maps in
$\mathbb{P}I$, where $Z_0=0$. Since $0 = \mathbb{P} 0$, it follows
that $Z= \mathbb{P}Y$, where $Y$ is a colimit of pushouts of
coproducts of maps in $I$. This proves the second statement. For the
first, take $X$, $Y$ and $Z$ as above, then
$$\mathbb{P}Y \smashprod_{\mathbb{P} \mathcal{I}} (-) \cong
Y \smashprod_{\mathcal{I}} (-).$$

We know that $Y$ is cofibrant in
$C((A,\Gamma)\mbox{-comod})$, hence by \cite[Proposition 2.1.4]{Hov04}
this functor preserves $\mathcal{P}$-equivalences.
The functor
$X \smashprod_{\mathbb{P} \mathcal{I}} (-)$ is a retract
of this functor and hence also preserves $\mathcal{P}$-equivalences.
\end{proof}

Focusing upon Franke's exotic model we compare the relative projective  model structure with the injective model structure on Franke's model.

\begin{proposition}\label{quillenstep}
The identity functor provides a Quillen adjoint pair between
$C^1(\mathcal{A})$ with the injective model structure and the
relative projective model structure. However, the two model
structures are not Quillen equivalent.
\end{proposition}

\begin{proof}
The identity is a left Quillen functor from $C(\mathcal{A})$ with the relative projective model structure to $C(\mathcal{A})$ with the injective model structure: a $\mathcal{P}$-cofibration is an injective cofibration as it is in particular a monomorphism by Lemma \ref{periodiccofib}. By \cite[2.1.5]{Hov04} a $\mathcal{P}$-equivalence is also a $\h_*$-isomorphism. It follows that we have a Quillen pair between the lifted model structures on $C^1(\mathcal{A})$.

However, to obtain a Quillen equivalence the weak equivalences have to agree between $\mathcal{P}$-cofibrant $X$ and injectively fibrant $Y$. We first show that this is not true for $C(\mathcal{A})$.
Consider the chain complex $X$ that consists of $E(1)_*$ concentrated in degree zero.
Then we take an injectively fibrant replacement $Y$ of $X$. The map $X \to Y$ is a quasi-isomorphism by definition, $X$ is $\mathcal{P}$-cofibrant and $Y$ injectively fibrant. This map is not a $\mathcal{P}$-equivalence.
To see this, just take $P=E(1)_*$ itself. Then $$\h_*(\mathcal{A}(P,X))=\mathcal{A}(P,E(1)_*)$$ concentrated in degree 0. But $$\h_*(\mathcal{A}(P,Y))=\textrm{Ext}_\mathcal{A}^*(E(1)_*,E(1)_*).$$ There are non-trivial higher Ext groups on the right side, so the two homologies are not isomorphic.
One can periodify the above to get the desired counterexample.
\end{proof}

\section{The quasi-projective model structure}\label{quasiprojective}

We saw at the end of Section \ref{relativeprojective} that the
relative projective model structure has fewer weak equivalences than
the injective model structure --- too few for the model categories to
be Quillen equivalent. To fix this deficit we can add weak equivalences to
the relative projective model structure via Bousfield localisation.
As a result we will obtain a model structure for $D^1(\mathcal{A})$ that
still has the nice monoidal properties of the relative projective
model structure.
For clarity, we restrict ourselves
to the case of a flat Adams Hopf algebroid
$(A,\Gamma)$, with a self equivalence $T$ on the category of comodules
which is compatible with the monoidal product. We will comment on the general
case at the end of this section.

We make use of the paper \cite{Bek00} to show that there is a model
structure on the category of chain complexes of comodules where the
cofibrations are the $\mathcal{P}$-cofibrations and the weak
equivalences are the quasi-isomorphisms. We state the theorem we
will use later, it is a theorem of Smith, but appears as
\cite[Theorem 1.7]{Bek00}.

The key is using the notions of a class of maps having the ``solution
set condition'' or being ``accessible''. It is technically
awkward to perform constructions such as Bousfield localisation
using a class of maps rather than a set of maps. However, if the
class of maps satisfies the solution set
condition, then it contains a set such that localising with respect to this set gives the
Bousfield localisation with respect to the whole class. So the solution set condition can be used to
avoid this awkwardness, for the full definition see \cite[Definition 1.5]{Bek00}. Accessibility is a
another technical condition \cite[Definition 1.14]{Bek00},
but in particular, an accessible class
of maps in a locally presentable category has the
solution set condition \cite[Proposition 1.15]{Bek00}.

\begin{theorem}\label{smith}
Let $\mathcal{C}$ be a locally presentable category, $\mathcal{W}$ a subcategory and
$I$ a set of morphisms of $\mathcal{C}$. Suppose they satisfy the criteria:
\begin{itemize}
\item c0: \ $\mathcal{W}$ is closed under retracts and has the 2-out-of-3 property
\item c1: \ $I\inj$ is contained in $\mathcal{W}$
\item c2: \ $I\cof \cap \mathcal{W}$ is closed under taking cell complexes
%transfinite composition and under pushouts
\item c3: \ $\mathcal{W}$ satisfies the solution set condition at $I$.
\end{itemize}
Setting the weak equivalences to be $\mathcal{W}$, the cofibrations to be $I\cof$
and the fibrations to be $(I\cof \cap W)\inj$,
one obtains a cofibrantly generated model structure on $\mathcal{C}$.
\end{theorem}

The notations $I\cof$ and $I\inj$ are technical
but standard, so we refer the reader to \cite[Subsection 2.1]{Hov99} rather
than recall them here.

We use this theorem to obtain a model structure on quasi-periodic chain complexes whose cofibrations are the $\mathcal{P}$-cofibrations as introduced in Section \ref{relativeprojective} and whose weak equivalences are the quasi-isomorphisms. Remember that our class of relative projectives $\mathcal{P}$ is constructed from the set of isomorphism classes of dualisable objects $\mathcal{S}$.

\begin{proposition}\label{quasimodel}
Let $\mathcal{W}$ be the set of quasi-isomorphisms and let $I$ be the set
$$I=\{ S^{n-1} P \to D^n P | P \in \mathcal{S}, \ n \in \mathbb{Z}  \}.$$
Then the above result gives a cofibrantly generated model structure
on $C((A,\Gamma)\mbox{-comod})$, which we call the quasi-projective model structure:
\begin{itemize}
\item the weak equivalences are the quasi-isomorphisms,
\item the cofibrations are the $\mathcal{P}$-cofibrations,
\item the fibrations are those maps that have the left-lifting-property with respect to the acyclic
cofibrations.
\end{itemize}
\end{proposition}
\begin{proof}
Condition c0 is obvious. The set $I$ is the set of generating
cofibrations of the relative projective model structure on
$C((A,\Gamma)\mbox{-comod})$, so $I\inj$ is the class of
acyclic $\mathcal{P}$-fibrations. Hence condition c1 is contained
within \cite[Proposition 2.1.5]{Hov04}  which states that every
projective equivalence is a homology isomorphism.

For condition c2, we know that $I\cof$ is closed under
transfinite composition and under pushouts. We know that the class
of monomorphisms that are quasi-isomorphisms is closed,
this class is the class of acyclic cofibrations of the injective model
structure. Hence their intersection is also closed. By the proof of
\cite[Proposition 3.13]{Bek00}, the quasi-isomorphisms are accessible,
thus by Proposition 1.15 of the same paper, the
solution set condition holds and we see that c3 holds.
\end{proof}

By Proposition \ref{create}, we also obtain a model structure on the
category of quasi-periodic chain complexes since we assumed that
$T(\mathcal{P})=\mathcal{P}$ for our chosen class of relative
projectives.

%where the weak equivalences are the quasi-isomorphisms and the cofibrations are generated by the set $\mathbb{P}I$, we also call this the \textbf{quasi-projective model structure}.

\begin{corollary}
There is a model structure on the category of quasi-periodic chain
complexes $C^{(T,N)}((A,\Gamma)\mbox{-comod})$ where the weak equivalences are
the quasi-isomorphisms and the cofibrations are degreewise split
monomorphisms with $\mathcal{P}$-cofibrant cokernel. We call this
the \textbf{quasi-projective model structure}.
\end{corollary}

\begin{corollary}
The quasi-projective model structure is the Bousfield localisation
of the relative projective model structure with respect to the class of
quasi-isomorphisms.
\end{corollary}

It should be remarked that a simpler way to construct the quasi-projective model structure can be found in \cite{Col06}. Here, Cole discusses how to construct a model structure on a category from ``mixing'' two existing ones. However, this does not examine whether the resulting model structure is cofibrantly generated, which is what we need to discuss its monoidal properties.

\bigskip

\bigskip

\bigskip

\begin{theorem}\label{uberness}
We have the following diagram of Quillen adjunctions, where all
vertical arrows are identity functors and the horizontal arrows
are periodification and the forgetful functor as introduced in
Proposition \ref{create}.
\[
\xymatrix@d@C+0.1cm{ C^{(T,N)}((A,\Gamma)\mbox{-comod})_{\textrm{rel proj}}
\ar@<1ex>[d]^(0.55){U}\ar@<-1ex>[r] &
C^{(T,N)}((A,\Gamma)\mbox{-comod})_{\textrm{quasi proj}}
\ar[l]\ar@<-1ex>[r]
\ar@<1ex>[d]^(0.55){U} &
C^{(T,N)}((A,\Gamma)\mbox{-comod})_{\textrm{inj}}
\ar@<1ex>[d]^(0.55){U}\ar[l] \\
(A,\Gamma)\mbox{-comod}_{\textrm{rel proj}}
\ar[u]^(0.45){\mathbb{P}}\ar@<-1ex>[r] &
(A,\Gamma)\mbox{-comod}_{\textrm{quasi proj}}
\ar[u]^(0.45){\mathbb{P}}\ar[l]\ar@<-1ex>[r] &
(A,\Gamma)\mbox{-comod}_{\textrm{inj}} \ar[l]
\ar[u]^(0.45){\mathbb{P}} }
\]
Furthermore the injective and quasi-projective model structures are
Quillen equivalent.
\end{theorem}

\begin{proof}
The upper vertical pairs are Quillen pairs as the cofibrations are the
same and a weak equivalence in the relative projective model
structure is a quasi-isomorphism. For the lower vertical pairs, a
cofibration in the quasi-projective model structure is a
$\mathcal{P}$-cofibration, hence a monomorphism. The weak
equivalences in both are the quasi-isomorphisms. Thus the identity
functor from the quasi-projective model structure to the injective
model structure preserves cofibrations and weak equivalences, hence
it is a left Quillen functor. This also shows that the
quasi-projective and injective model structures must be Quillen
equivalent as they have the same weak equivalences.
\end{proof}

We note that one could have constructed the quasi-projective model
structure on $C^{(T,N)}((A,\Gamma)\mbox{-comod})$ directly, taking care to show
the category-theoretic conditions of Theorem \ref{smith}.
We are now going to exploit the monoidal properties of this model structure.
We would like to make use of Proposition \ref{modulemodel}, but while the pushout product axiom holds, it is not obvious to us why the monoid axiom would hold for this model structure on $C((A,\Gamma)\mbox{-comod})$ or $C^{(T,N)}((A,\Gamma)\mbox{-comod})$. Thus we prove that we have a monoidal model structure directly. 

\begin{lemma}
The quasi-projective model structure on
$C((A,\Gamma)\mbox{-comod})$ is monoidal.
\end{lemma}
\begin{proof}
We first note that the unit is cofibrant.
We know that the pushout of two cofibrations is a cofibration,
because the relative projective model structure satisfies the
pushout product axiom. Now consider the pushout product of a
generating cofibration with a generating acyclic cofibration. Let
$k$ be the inclusion $S^{n-1}\mathcal{I} \to D^n \mathcal{I}$, then
$P \smashprod k$ is the general form of any generating cofibration
(where $P$ is a dualisable object). Let $f \co X \to Y$ be a
generating acyclic cofibration. Then the pushout product of $P
\smashprod k$ and $f$ is simply $P$ smashed with the pushout product
of $k$ and $f$. We must check that this map is a quasi-isomorphism.
Consider the following pushout diagram
$$
\xymatrix{
S^{n-1} \mathcal{I} \smashprod X \ar[r] \ar[d] &
D^{n} \mathcal{I} \smashprod X \ar[d] \\
S^{n-1} \mathcal{I} \smashprod Y \ar[r] &
Q \\
}
$$
the left hand vertical map is a monomorphism and a homology
isomorphism (modulo signs for the differential, it is just a
suspension of $f$). It follows that the right hand vertical map is
also a monomorphism and a homology isomorphism as acyclic cofibrations
are preserved by pushouts.

It is easy to see that $D^{n} \mathcal{I} \smashprod X$ has trivial homology, as does
$D^{n} \mathcal{I} \smashprod Y$, whence the pushout product of $k$ and $f$ (see page \pageref{box}),
$$k \square f \co Q \to D^{n} \mathcal{I} \smashprod Y$$ must be a homology isomorphism. We now need to see that
$(k \square f) \smashprod P$ is a homology isomorphism. But this is a statement about underlying
$A$-modules, where $P$ is finitely generated and projective, hence $(k \square f) \smashprod P$ is a homology isomorphism.
\end{proof}

\begin{corollary}
The pushout product axiom holds for $C^{(T,N)}((A,\Gamma)\mbox{-comod})$ with
the quasi-projective model structure and monoidal product
$\smashprod_{\mathbb{P}\mathcal{I}}$.
\end{corollary}
\begin{proof}
We copy the proof of the above lemma.  We know
that the relative projective model structure is monoidal, hence the
pushout product of two cofibrations in the quasi-projective model
structure is a cofibration. A generating acylic cofibration for this
model structure on $C^1(\mathcal{A})$ has form $\mathbb{P} f
\co \mathbb{P} X \to \mathbb{P} Y$, where $f$ is a  generating
acyclic cofibration for $C^1(\mathcal{A})$. Similarly, a
generating cofibration looks like $\mathbb{P}(P \smashprod k)$, for $k$
and $P$ as in the previous proof. Recall that now we are taking the
product over $\mathbb{P} \mathcal{I}$. Drawing the relevant
diagram it is easy to see that we need $\mathbb{P}( (k \square f)
\smashprod P) $ to be a homology isomorphism. We know that $(k \square
f) \smashprod P$ is a homology isomorphism and $\mathbb{P}$ preserves
homology isomorphisms, since it is just an infinite direct sum of
shifts and applications of $T$, so we are done.
\end{proof}

Recall that a fibration in the relative projective model structure
is, in particular, a surjection \cite[Proposition 2.1.5.]{Hov04}. It
follows that any quasi-projective fibration is a surjection and
similarly that any quasi-projective cofibration is a monomorphism.

\begin{lemma}
The quasi-projective model structure is proper.
\end{lemma}
\begin{proof}
The long exact sequence in homology implies that
any model structure on $C((A,\Gamma)\mbox{-comod})$ will be proper as long as weak equivalences coincide
with quasi-isomorphisms, every cofibration is an injection, and every fibration is a
surjection. It follows immediately that the quasi-projective model structure on
$C^{(T,N)}((A,\Gamma)\mbox{-comod})$ is proper.
%\cite{Hover modelshaves}
\end{proof}

We summarise our work in the following theorem.

\begin{theorem}\label{mainresult}
The quasi-projective model structure on $C^{(T,N)}((A,\Gamma)\mbox{-comod})$ is cofibrantly generated, proper and monoidal.
In the special case of $T$, $\mathcal{A}$ and $N$ as in Section \ref{Franke} the homotopy category of this model category is precisely $D^1(\mathcal{A})$.
\end{theorem}

\begin{corollary}
Franke's model $D^1(\mathcal{A})$ is a symmetric monoidal category
with tensor product $\smashprod^L_{\mathbb{P}\mathcal{I}}$.
\end{corollary}

\begin{rmk}
Consider a general Grothendieck category that is closed monoidal.
Assume first that that the collection of isomorphism classes of
dualisable objects forms a set $\mathscr{S}$. Secondly, assume that
this set generates the category, that is, the coproduct of all
elements of $\mathscr{S}$ is a generator. Then the relative
projective, quasi-projective and injective model structures all
exist for $C^{(T,N)}(\abcat)$. The first two are monoidal model
categories, all three are proper
 and the obvious analogue of
Theorem \ref{uberness} holds.

The extra work required is reproving \cite[Theorem 2.1.5]{Hov04}
in this situation, this is quite straightforward.
The fibrations in the relative projective model structure are all epimorphisms,
so both model structures are proper.  Since dualisables are flat
the pushout product axioms hold.
\end{rmk}

\section{Picard groups}\label{Picard}

In this section we are going to compare the Picard group of
$D^1(\mathcal{A})$ to the Picard group of the $K_{(p)}$-local stable
homotopy category $\Ho(L_1 \mathcal{S})$.
Let $\mathcal{M}$ be a monoidal category with unit $\mathcal{I}$ and product
$\wedge$. The \textbf{Picard group} $\pic(\mathcal{M})$ is the group
of invertible objects in this category: its objects are the
isomorphism classes of $X \in \mathcal{M}$ such that there is an
object $Y \in \mathcal{M}$ with $X \wedge Y \cong \mathcal{I}$. The group
multiplication is induced by $\wedge$.

Picard groups have their origin in algebraic geometry but have
increasingly been studied in stable homotopy theory. Of particular
interest are the Picard groups of Bousfield localisations of the stable
homotopy category or the homotopy category of $R$-modules for an
$E_\infty$-ring spectrum $R$. For example, it is well-known that the
Picard group of the stable homotopy category is $\mathbb{Z}$,
generated by the 1-sphere. This result was later reproved by Baker and
Richter in \cite{BakRic05} who also gave computations of
$\Ho(R-\mbox{mod})$ for some connective $E_\infty$ ring spectra
$R$.

Let $\Ho(L_n \mathcal{S})$ denote the $E(n)$-local stable homotopy
category where $E(n)$ is the $n^{th}$ Johnson-Wilson spectrum. Hovey
and Sadofsky showed that for $n^2 + n < 2p-2 $, $$\pic(\Ho(L_n
\mathcal{S})) \cong \mathbb{Z},$$ consisting of shifts of the sphere
spectrum, see \cite{HovSad99}. Georg Biedermann, in \cite{Bie07}, later
extended the computation to $p$ and $n$ with $ p > n+1$ and $4p-3 >
n^2 + n$. This means we know that for $p$ odd, $\pic(\Ho(L_1
\mathcal{S})) \cong \mathbb{Z}$, consisting of the spheres.

The previous section shows that
that $D^1(\mathcal{A})$ is a symmetric monoidal category,
so it makes sense to consider its Picard group and compare it to $\pic(\Ho(L_1 \mathcal{S}))$.

Let us remember that Franke's model $C^1(\mathcal{A})$ does not only work for the $E(1)$-local stable homotopy category and $p>2$. An analogous construction works for all $n$ and $p$ with $n^2 + n < 2p-2$. Hence in this range, the $E(n)$-local stable homotopy category possesses an exotic algebraic model. Although not obviously related, it is no coincidence that this range for $n$ and $p$ agrees with the range of Hovey's and Sadofsky's Picard group computation. 

Both results use the fact that the $E(n)$-based Adams spectral sequence collapses for those $n$ and $p$. In Franke's proof, the collapsing is used rather indirectly for an algebraic description of some morphisms in $\Ho(L_n \mathcal{S})$. Hovey and Sadofsky show that an element $X$ of $\pic(\Ho(L_n \mathcal{S}))$ satisfies
\[
E(n)_* \cong E(n)_*(X) \quad\mbox{in}\quad E(n)_*E(n)\mbox{-comod}.
\]
Sparseness of the $E(n)$-Adams spectral sequence is a key ingredient for constructing a weak equivalence $L_nS^0 \longrightarrow X$.

While in the ``exotic range'' $n, p$ with $n^2 + n < 2p-2$ the $E(n)$-local Picard group is trivial, this is not the case for $n=1$ and $p=2$. For $p=2$, $$\pic(\Ho(L_1 \mathcal{S})) \cong \mathbb{Z} \oplus \mathbb{Z}/2.$$ The $\mathbb{Z}/2$-summand is generated by the so-called question mark complex \cite[Theorem 6.1]{HovSad99}. Also, we know that for $p=2$, $\Ho(L_1 \mathcal{S})$ is rigid and hence has no exotic models. It would be an interesting topic to relate Picard groups to the existence of exotic models. 

\begin{lemma}
$\pic(D^1(\mathcal{A}))$ is a set.
\end{lemma}

\begin{proof}
Our category $D^1(\mathcal{A})$ is triangulated, so we apply 
\cite[A.2.8, 2.1.3 and 2.3.6]{HovPalStr97}.
\end{proof}

Franke's theorem tells us that $D^1(\mathcal{A})$ and $\Ho(L_1
\mathcal{S})$ are equivalent as triangulated categories via the
functor $\mathcal{R}$. However,
$\mathcal{R}:D^1(\mathcal{A}) \longrightarrow \Ho(L_1 \mathcal{S}) $
is not monoidal as it is not associative \cite[Remark 1.4.2]{Gan07}.
Hence it does not automatically induce a group homomorphism between the respective
Picard groups. Extra work is needed to see that
$\mathcal{R}$ preserves just enough
structure to use it for comparing these Picard groups.

\begin{theorem}[Ganter]\label{ganteriso}
There is a natural isomorphism $$\mathcal{R}(C
\smashprod^L_{\mathbb{P}\mathcal{I}} D) \cong \mathcal{R} ( C ) \wedge^L
\mathcal{R}( D ) $$ where $\wedge^L$ denotes the smash product in
$\Ho(L_1 \mathcal{S}) $, $\mathcal{I}=E(1)_*$
and $\smashprod^L_{\mathbb{P}\mathcal{I}}$ is the
monoidal product of $D^1(\mathcal{A})$.
\end{theorem}

Note that in her theorem, Ganter denotes the derived tensor product
of quasi-periodic chain complexes by $\otimes^L_{E(1)_*}$.
This is not to be confused with the tensor product in $D(\mathcal{A})$.

She defines the monoidal product on $D^1(\mathcal{A})$
as the tensor product of underlying degreewise flat replacements,
i.e. flat as $E(1)_*$-modules. This construction is not accompanied by a rigorous proof, so it is not clear to the authors to what extent this equips $D^1(\mathcal{A})$ with a monoidal structure.
While Ganter mentions the concept of monoidal model categories, she does not address the question of whether $C^1(\mathcal{A})$ is such a category.

Theorem \ref{mainresult} closes this gap in the following way.
For a monoidal model category $\mathcal{M}$
with product $\otimes$, the derived product
$\otimes^L$ on $\Ho(\mathcal{M})$ is defined as $$X \otimes^L Y = QX
\otimes QY$$ where $QX$ and $QY$ are cofibrant replacements of $X,Y
\in \mathcal{M}$ \cite[4.3.2]{Hov99}. Since in our case the cofibrant objects
are also degreewise flat, we are consistent with Ganter's result and can write down the above theorem.
Further, we can use it to compute $\mathcal{R}(C \smashprod^L_{\mathbb{P}\mathcal{I}} D)$.

\begin{theorem}
$\pic(D^1(\mathcal{A})) \cong \mathbb{Z}$
\end{theorem}

\begin{proof}
Let $C$ and $D$ be an inverse pair of quasi-periodic chain complexes in
$D^1(\mathcal{A})$, so
$$C \smashprod^L_{\mathbb{P}\mathcal{I}} D \cong
\mathbb{P}\mathcal{I}.$$
Applying $\mathcal{R}$ to this equation and
using Ganter's theorem gives 
$$\mathcal{R}(C
\smashprod^L_{\mathbb{P}\mathcal{I}} D) \cong
\mathcal{R}(\mathbb{P}\mathcal{I}) \cong \mathcal{R}(C) \wedge^L
\mathcal{R}(D).$$ Furthermore, we know that $\mathcal{R}$ sends the
unit $\mathbb{P}\mathcal{I}$ to the $E(1)$-local sphere $L_1 S^0$.
This can be read off the natural isomorphism given in Franke's
theorem and the fact that $E(1)_*$ reflects isomorphisms. So we
arrive at the statement
$$L_1 S^0 \cong \mathcal{R}(C) \wedge^L \mathcal{R}(D).$$
This means that $\mathcal{R}(C)$ and  $\mathcal{R}(D)$ are in the
Picard group of $\Ho(L_1 \mathcal{S})$ and hence must be suspensions
of the $E(1)$-local sphere. Being an equivalence of triangulated
categories $\mathcal{R}$ reflects isomorphisms, so $C$ and $D$ are
shifts of $\mathbb{P}\mathcal{I}$. Since
$$\mathbb{P}\mathcal{I}[i]
\smashprod^L_{\mathbb{P}\mathcal{I}} \mathbb{P}\mathcal{I}[j] \cong
\mathbb{P}\mathcal{I}[i+j]$$
in $D^1(\mathcal{A})$, this completes
the proof of our theorem.
\end{proof}

As mentioned, the construction of Franke's functor $\mathcal{R}$ also extends to
the $E(n)$-local stable homotopy category for $n^2+n\leq 2p-2$, i.e.
there is an equivalence of triangulated categories
\[
\mathcal{R}: D^1(\mathcal{A}) \longrightarrow \Ho(L_n \mathcal{S})
\]
for some abelian category $\mathcal{A}$ that is equivalent to the
category of $E(n)_*E(n)$-comodules. But the monoidal behaviour
of $\mathcal{R}$ in this general case is not yet known.

In particular, Ganter's construction only works for $n=1$. It will be worth investigating whether our results about the properties of $C^1(\mathcal{A})$ give a more straightforward analogue of Theorem \ref{ganteriso}. 

Ganter's isomorphism for $n=1$ allows us to read off the Picard group of Franke's model in a simple way. However, it would be interesting to know if, particularly for higher $n$, this Picard group can be calculated more directly. We hope that considering these questions will lead to more insight into the existence of exotic models and hence understanding the structure of the $E(n)$-local stable homotopy categories.

\end{document}